\documentclass[a4paper,12pt,twoside,notitlepage,reqno]{amsart}

\usepackage{amsmath,amssymb,amsbsy,amsthm, comment}
\usepackage[hmargin=1.4in,vmargin=1.4in]{geometry}
\usepackage{url}

\usepackage{tikz}
\usetikzlibrary{matrix}
\usepackage[linktocpage]{hyperref}
\hypersetup{
    colorlinks=true,        
    linkcolor=red,          
    citecolor=red,         
    filecolor=magenta,      
    urlcolor=cyan           
}

\addtocontents{toc}{\protect\setcounter{tocdepth}{1}}

\usepackage{eucal}
\usepackage[all]{xy}

\newcommand{\Frak}{ \operatorname{Frac} }

\newcommand{\F}{ \mathbb{F} }

\newcommand{\dra}{\dashrightarrow}
\newcommand{\lra}{\longrightarrow}

\newcommand{\bG}{\mathbb{G}}

\theoremstyle{plain}
	\newtheorem{thm}{Theorem}
	
		\numberwithin{thm}{section}
	\newtheorem{lemma}[thm]{Lemma}
	\newtheorem{prop}[thm]{Proposition}

	\newtheorem{claim}[thm]{Claim}

	\newtheorem*{thm*}{Theorem}
	\newtheorem*{lemma*}{Lemma}
	\newtheorem*{prop*}{Proposition}
	\newtheorem*{cor*}{Corollary}
	\newtheorem*{conj*}{Conjecture}

\theoremstyle{definition}
	
	\newtheorem*{example*}{Example}
	
	\newtheorem{remark}[thm]{Remark}

\begin{document}


\title[Rational self-maps on semiabelian varieties]{Rational self-maps with a regular iterate on a semiabelian variety}

\author{Jason Bell}
\address{University of Waterloo \\
Department of Pure Mathematics \\
Waterloo, Ontario \\
N2L 3G1, Canada}
\email{jpbell@uwaterloo.ca}

\author{Dragos Ghioca}
\address{University of British Columbia\\
Department of Mathematics\\
Vancouver, BC\\
V6T 1Z2, Canada}
\email{dghioca@math.ubc.ca}

\author{Zinovy Reichstein}
\address{University of British Columbia\\
Department of Mathematics\\
Vancouver, BC\\
V6T 1Z2, Canada}
\email{reichst@math.ubc.ca}

\thanks{The authors were partially supported by Discovery Grants from the National Science and Engineering Research Council of Canada.}

\begin{abstract}
Let $G$ be a semiabelian variety defined over an algebraically closed field $K$ of characteristic $0$. Let $\Phi\colon G\dashrightarrow G$ be a dominant  rational self-map. Assume that an iterate $\Phi^m \colon G \to G$ is regular for some $m \geqslant 1$ and that there exists no non-constant homomorphism $\tau\colon G\lra G_0$ of semiabelian varieties such that $\tau\circ \Phi^{m k}=\tau$ for some $k \geqslant 1$. We show that under these assumptions $\Phi$ itself must be a regular. We also prove a variant of this assertion in prime characteristic
and present examples showing 
that our results are sharp. 
\end{abstract}

\maketitle



\section{Introduction}

When dealing with questions involving algebraic dynamical systems, it is often easier---and at times necessary---to work with an iterate of one's self-map rather than with the map itself.  This technique, for instance, is often employed in complex analytic or $p$-adic analytic uniformization results (see, for example, \cite{Sarah, Jason, BGT0, GT0, H-Y, River}), where one replaces one's map with an iterate so that a point that is periodic after reduction modulo a suitable prime becomes a fixed point.  To give another example, in Tits' \cite{Tit72} famous proof that finitely generated linear groups either contain a free non-abelian subgroup or are solvable-by-finite, one views the matrices as giving self-maps of a suitable projective space and works with a suitable iterate of a well-chosen element of the group along with a suitable conjugate in order to invoke the Ping-Pong lemma.

In this sense there is a kind of unspoken principle in arithmetic dynamics that a suitably chosen iterate of a rational self-map often becomes better behaved in some important respect, and knowing this then allows one to obtain information about one's original map.  For this reason, it is natural to ask about the relationship between an iterate of a self-map of an algebraic variety having a given desirable property and the map itself having this property.  


In this paper, we investigate what can be said when a rational self-map has some iterate that is a regular morphism. It is easily seen that one cannot in general deduce that the original map is regular.  For example, the Noether self-map of $\mathbb{P}^2$ given by $[x:y:z]\mapsto [1/x:1/y:1/z]$ is not regular, but its second iterate is the identity map.  On the other hand, one expects that for varieties that are sufficiently ``rigid'' one should be able to draw strong conclusions about a rational self-map that has a regular iterate.  

Of particular interest to us is the relationship between regularity of an iterate of a rational self-map and regularity of the original map for self-maps of semiabelian varieties.  We recall that a semiabelian variety is an algebraic group $G$ that fits into a short exact sequence of algebraic groups
\begin{equation}
\label{eq:semiabelian}
1\lra \bG_m^N\lra G\lra A\lra 1,
\end{equation}
for some non-negative integer $N$ and some abelian variety $A$.  Algebraic dynamical systems on semiabelian varieties have arisen in many classical conjectures (see, for example, \cite{DML-book, DMM-1, DMM-2, Zhang}) and enjoy a long history, due to the fact that finitely generated subgroups can be viewed as orbits of the identity under a group of translation maps. In addition, semiabelian varieties are rigid in the important respect that they do not allow non-constant maps from $\mathbb{A}^1$ (see also Remark \ref{rem2} for the connection of this property and the existence of dominant non-regular self-maps with some regular iterate).  For these reasons it is natural to consider the question of the relationship between regularity of an iterate of a rational self-map and regularity of the original map in this more restricted context. 




In the setting of semiabelian varieties, we are able to give an essentially complete characterization of when dominant rational self-maps have a regular iterate. 

\begin{thm}
\label{thm:main}
Let $G$ be a semiabelian variety defined over an algebraically closed field $K$, and let $\Phi \colon G \dra G$ be a dominant rational self-map defined over $K$. Suppose 

\begin{itemize}
\item
an iterate $\Phi^m$ is a regular map $G \to G$ for some $m \geqslant 1$, and

\item
there does not 
exists a non-constant homomorphism of semiabelian varieties $\chi \colon G \to G_0$ such that $\chi \circ \Phi^{mk} =\chi$ for some $k \geqslant 1$.
\end{itemize}

\smallskip
(a) If $\operatorname{char}(K) = 0$, then $\Phi$ is itself a regular map.

\smallskip
(b) If $\operatorname{char}(K) = p > 0$, then assume additionally that 
there does not exist a semiabelian variety $G_1$ defined over a finite subfield $\F_q$ of $K$ along with a non-constant group homomorphism of semiabelian varieties $\chi \colon G \to G_1$ such that $\chi \circ \Phi^{mk} = F^r\circ \chi$ for any positive integers $k$ and $r$, where $F:G_1\lra G_1$ is the Frobenius endomorphism corresponding to the finite field $\F_q$.  
Then $\Phi$ is a regular map.
\end{thm}

We note that any rational map of abelian varieties is regular; indeed, a rational map from a semiabelian variety to an abelian variety is always regular.  Thus Theorem \ref{thm:main} is vacuous when $G$ is an abelian variety.  Nevertheless, our main tool in proving Theorem \ref{thm:main}, Proposition \ref{prop:invariant} below, is actually non-trivial to verify for abelian varieties.   

\subsection{Proof strategy}

Our proof of Theorem~\ref{thm:main} will be divided into two steps given by the two propositions below.
Proposition~\ref{cor:indeterminacy} yields the existence of an irreducible  subvariety $W\subset G$ of codimension $1$, which is totally invariant under a suitable iterate $\Phi^{mk}$. Proposition~\ref{prop:invariant} then finishes the proof of Theorem~\ref{thm:main}.

\begin{prop}
\label{cor:indeterminacy}
Let $G$ be a semiabelian variety over an algebraically closed field $K$ and $\Phi \colon G \dasharrow G$ be a dominant rational self-map.
Suppose $\Phi$ is not regular but an iterate $\Phi^m$ is regular for some $m \geqslant 2$. Then there exists an irreducible 
codimension $1$ subvariety $W$ of $G$ which is totally invariant under a suitable iterate $\Phi^{mk}$, for some integer $k \geqslant 1$.
\end{prop} 

Here and throughout our paper, we will say that a closed subvariety $W\subset X$ is totally invariant under 
a dominant self-map $\Psi \colon X \to X$ if $W = (\Psi)^{-1}(W)$. The last equality should be understood set-theoretically: 
if $\Psi(x) \in W$, then $x \in W$ for any $x \in X(K)$. For example, if $\Psi \colon \mathbb{A}^1 \lra \mathbb{A}^1$ 
is given by $x \mapsto x^{2}$ and $W$ is the origin in $\mathbb{A}^1$, cut out by $x = 0$,  then the preimage
$(\Psi)^{-1}(W)$ is cut out by $x^{2} = 0$. So, $W = \Psi^{-1}(W)$ set-theoretically but not scheme-theoretically.

\begin{prop}
\label{prop:invariant}
Let $G$ be a semiabelian variety defined over an algebraically closed field $K$, endowed with a dominant regular self-map $\Psi$. Suppose there exists an irreducible codimension $1$ subvariety $W\subset G$, which is totally invariant under $\Psi$.

\smallskip
(a) Assume further that $\operatorname{char}(K) = 0$. Then there 
exists an integer $k\geqslant 1$ and a semiabelian variety $G_0$ defined over $K$ along with a non-constant group homomorphism of semiabelian varieties  $\chi \colon G \to G_0$ such that $\chi \circ \Psi^k =\chi$. 

\smallskip
(b) Assume now that $\operatorname{char}(K) = p > 0$. Then at least one of the following two statements must hold: 
\begin{itemize}
\item either there 
exists an integer $k\geqslant 1$ and a semiabelian variety $G_0$ defined over $K$ along with a non-constant group homomorphism of semiabelian varieties  $\chi \colon G \to G_0$ such that $\chi \circ \Psi^k =\chi$. 
\item or there exist integers $k, r \geqslant 1$, and a semiabelian variety $G_1$ defined over a finite subfield $\F_q$ of $K$ along with a non-constant group homomorphism of semiabelian varieties $\chi \colon G \to G_1$ such that $\chi \circ \Psi^k = F^r\circ \chi$, where $F \colon G_1\lra G_1$ is the Frobenius 
endomorphism and $F^r$ acts trivially on the finite field $\F_q$. 
\end{itemize}
\end{prop}

\subsection{Notational conventions}

Throughout this note, $K$ will denote an algebraically closed field, $G$ a semiabelian variety defined over $K$ and $\Phi \colon G \dasharrow G$ a rational dominant self-map of $G$ (as a $K$-variety).

It is well known that $G$ is an abelian group and every regular dominant self-map $G \to G$ of $G$ (again, as a $K$-variety) is a composition of a translation and a group endomorphism; see Lemma~\ref{lem3.2}. 
For background material on semiabelian varieties we refer the reader to~\cite{NW} and (in prime characteristic) to~\cite{Iitaka}. 
For any $n \geqslant 0$, we will denote the $n$-th iterate of $\Phi$ by $\Phi^n \colon G \dasharrow G$. As usual, by
$\Phi^0$ we mean the identity map $G \to G$.

\subsection{Structure of the paper}
The remainder of this paper is structured as follows. Section~\ref{sec:remarks} contains 
examples and remarks motivating Theorem~\ref{thm:main}~and Proposition~\ref{prop:invariant}.
After some preliminary lemmas in Section~\ref{sect.preliminaries}, we prove
Proposition~\ref{cor:indeterminacy} in Section~\ref{sec:indeterminacy} and Proposition~\ref{prop:invariant} 
in Section~\ref{sec:proof}. As we pointed out above, our main result, Theorem~\ref{thm:main},
is a direct consequence of these two propositions.

\section{Motivating remarks}
\label{sec:remarks}

The purpose of this section is to motivate Theorem~\ref{thm:main}~and Proposition~\ref{prop:invariant} 
and show that their assumptions cannot be weakened.

\subsection{Remarks on Theorem~\ref{thm:main}}
\label{subsec:examples}

\begin{remark} \label{rem2}
Theorem~\ref{thm:main}(a) fails if we replace $G$ by an arbitrary irreducible smooth variety $X$, even if we assume $X$ to be affine. 
To see this, set $X=\mathbb{A}^1$ (defined over a field of characteristic $0$)  and define $\Phi \colon\mathbb{A}^1 \dra \mathbb{A}^1$ by  $\Phi(x)=x^{-2}$. In this case $\chi \circ \Phi^{k} \neq \chi$ for any non-constant morphism $\chi \colon \mathbb A^1 \to X_0$ 
and any $k \geqslant 1$. Moreover, $\Phi^2$ is regular on $\mathbb{A}^1$ but $\Phi$ is not regular.

\smallskip
The special properties of semiabelian varieties $G$ used in the proof of Theorem~\ref{thm:main} are the following.

\smallskip
(i) We understand regular self-maps on $G$: see Lemma~\ref{lem3.2}.

\smallskip
(ii) We understand invariant subvarieties of $G$ under the action of a regular self-map; see~\cite{P-R}.

\smallskip \noindent
Neither (i) nor (ii) is available if we replace $G$ by a general smooth affine variety $X$.
\end{remark}

\begin{remark} \label{rem3}
In Theorem~\ref{thm:main}(a) the assumption that there is no non-constant morphism $ \chi \colon G \to G_0$ of semiabelian varieties such that $\chi \circ \Phi^{km} = \chi$ for 
some $k \geqslant 1$ is indeed necessary. To see this, suppose $G=\bG_m$ in Theorem~\ref{thm:main} and $\Phi \colon \bG_m \dashrightarrow \bG_m$ is given by
$$x\mapsto \frac{-x}{x+1}.$$
Then $\Phi^2$ is the identity on $\bG_m$. Hence $\Phi^2$ is regular even though $\Phi$ is not, and
$\chi \circ \Phi^2 = \chi$ when $\chi = {\rm id} \colon G \to G$.
\end{remark}

\begin{remark}
\label{rem5}
Working with $G=\bG_m$ once again, we see that the assumption that $\Phi^{mk}$ does not preserve a non-constant homomorphism $\chi \colon \bG_m \to G_0$ of semiabelian varieties for any $k \geqslant 1$, as in part (a), is not enough to guarantee that $\Phi$ is regular in part (b).
For example, in characteristic $p$ consider the rational map $\Phi \colon \bG_m \dashrightarrow \bG_m$ given by $x \mapsto x^p + 1$. 
Then $\chi\circ  \Phi^{mk} \neq \chi$ for any $k \geqslant 1$ and any
non-constant morphism $\chi \colon \mathbb G_m \to G_0$ of semiabelian varieties. (Indeed, here we may assume without loss of generality
that $G_0 = \mathbb G_m$ 
and $\chi \colon t \to t^d$ for some integer $d \neq 0$.)
On the other hand, 
$\Phi^p \colon x \mapsto x^{p^p}$ is regular but $\Phi$ is not. Note, however, that if $\chi \colon \mathbb G_m \to \mathbb G_m$
is the identity map, then $\chi \circ \Phi^p = \chi^{p^p}$.  
\end{remark}

\begin{remark}
\label{rem1.5}
The assumption of Theorem~\ref{thm:main}(a)
that there does not 
exists a nonconstant group homomorphism of semiabelian varieties $\chi \colon G \to G_0$ such that $\chi \circ \Phi^{mk} =\chi$ for any $k \geqslant 1$, is readily seen to be satisfied if there exists a point $\alpha\in G(K)$ with a Zariski dense orbit under $\Phi^m$.
Indeed, if $\chi \circ \Phi^{mk} = \chi$ for some $k \geqslant 1$, then $\chi$ can assume only finitely many values 
on the orbit of $\alpha$. If $\chi \colon G\lra G_0$ is non-constant, this orbit cannot be dense in $G$.
\end{remark}

\begin{remark} \label{rem1.6}
Remark~\ref{rem1.5} relates our question to the Zariski dense orbit conjecture, due to Amerik-Campana \cite{A-C} and  Medvedev-Scanlon \cite{M-S} (both generalizing an earlier conjecture of Zhang \cite{Zhang}). Consider a dominant rational self-map $\Phi:X\dra X$, where $X$ is an arbitrary variety defined over an algebraically closed field $K$ of characteristic $0$. Assume that there exists no non-constant rational map $f:X\dra \mathbb{P}^1$ such that $f\circ \Phi=f$.
The Zariski dense orbit conjecture predicts that in this situation there exists a point $\alpha\in X(K)$ whose 
orbit under $\Phi$ is well-defined and Zariski dense in $X$. This conjecture has been established
in several special cases, e.g., where $X$ is a semiabelian variety (see \cite{G-Scanlon, G-Matt, G-S}), 
or a smooth variety of positive Kodaira dimension (see \cite{BGRS}), but it remains open in general
over countable fields $K$ of characteristic $0$. 

For varieties $X$ defined over algebraically closed fields $K$ of characteristic $p>0$, the Zariski dense orbit conjecture is even more subtle. Here one needs to modify the statement to take into account the possible action of the Frobenius on varieties defined over finite fields. 
In fact, the characteristic $p$ analogue of the Zariski dense orbit conjecture stated 
in~\cite{G-S-2} involves a condition similar to the one appearing in part~(b) of our Theorem~\ref{thm:main}. 
Only sporadic results are known in positive characteristic (e.g., the case when $K$ is uncountable 
is handled in~\cite{BGR} and the case of an endomorphisms of $\bG_a^N$ defined over a finite field in \cite{G-S-3}).
\end{remark}

\subsection{Remarks on Proposition~\ref{prop:invariant}}
\label{subsec:examples_2}

\begin{remark} \label{rem.prop0}
Invariant subvarieties of semiabelian varieties under a dominant self-map $\Psi:G\lra G$ have been extensively studied, starting with the work of M.~Hindry \cite{Hindry} (where $\Psi$ is the multiplication-by-$l$ map on $G$) and culminating with the thorough description for all dominant regular maps $\Psi$ in arbitrary characteristic, due to
R.~Pink and D.~Rossler~\cite{P-R}. Invariant subvarieties appear naturally in the main conjectures in arithmetic dynamics, such as the Dynamical Mordell-Lang Conjecture (see \cite{DML-book}) and the Dynamical Manin-Mumford Conjecture (see \cite{DMM-1, DMM-2}). Totally invariant subvarieties come up in the study of the algebraic degree of a self-map $\Psi \colon G \to G$  (see \cite{BG-1}) and are related to arithmetic properties of $\Psi$; see \cite{Cantat, BGT-2, BGT-3}. 
\end{remark}


\begin{remark}
\label{rem:1}
The following example shows that we cannot expect the conclusion of Proposition~\ref{prop:invariant} to hold if we only assume that 
$W$ is invariant under $\Psi$ (i.e., $\Psi(W)=W$) instead of totally invariant (i.e., $\Psi^{-1}(W)=W$).

Let $K$ be a field of characteristic $0$, $\Psi:\bG_m^2\lra \bG_m^2$ be and automorphism of $\bG_m^2$ given by 
$\Psi(x_1,x_2)=\left(x_1^2,x_2^3\right)$,
and $W$ be the codimension $1$ closed subgroup $\{ 1 \} \times \bG_m$. In other words, $W$ is cut out by $x_1 = 1$.
Then $W$ is invariant under $\Psi$ but is not totally invariant. On the other hand,
it follows from Remark~\ref{rem1.5} that
$\chi \circ \Psi^k \neq \chi$ for any non-constant homomorphism $\bG_m^2 \to G_0$ of semiabelian varieties. 
\end{remark}

\begin{remark}
\label{rem:2}
The next example shows that Proposition~\ref{prop:invariant} fails if we only assume that there exists a totally invariant subvariety of positive dimension (but not necessarily of codimension $1$).

Let $K$ be a field of characteristic $0$ and $\Psi \colon \bG_m^4\lra \bG_m^4$ be an automorphism of $\bG_m^4$ given by
$$\Psi(x_1,x_2,x_3,x_4)=\left(x_1^2x_2,x_1^3 x_2^2,x_3^2 x_4, x_3^3 x_4^2\right).$$
The surface $S$ given by $x_1=x_3$ and $x_2=x_4$ is totally invariant, but $\chi \circ \Psi^k \neq \chi$ for any non-constant morphism $\bG_m^2 \to G_0$ of semiabelian varieties, for any $k \geqslant 1$. 
So, the hypothesis in Proposition~\ref{prop:invariant} that $W$ has codimension $1$ is indeed crucial. 
\end{remark}


\section{Preliminaries}
\label{sect.preliminaries}

In this section we collect three preliminary lemmas. These will be used in the proofs of Propositions~\ref{cor:indeterminacy} and~\ref{prop:invariant}. The first lemma concerns  maps of schemes of finite type, 
the other two are about self-maps $G \to G$ of semiabelian varieties.

\begin{lemma}
\label{lem3.1}
Let $\Psi \colon X \to Y$ be a morphism of integral schemes of finite type over $K$. Let $x$ be a $K$-point of $X$,
$y = \Psi(x)$, and let $\widehat{\mathcal{O}}_x$ and $\widehat{\mathcal{O}}_y$ denote the complete local rings at $x$ and $Y$, 
respectively. If $A:=\Psi^* (\widehat{\mathcal{O}}_y )$, viewed as a subring of $\widehat{\mathcal{O}}_x$, then

\smallskip
(a) If $\Psi$ is \'etale, then $\Frak(A) \cap \widehat{\mathcal{O}}_x = A$.

\smallskip
(b) Moreover, assume that $\Psi$ can be written as a composition $\Psi_{\rm et} \circ \Psi_{\rm insep}$, where
$\Psi_{\rm insep} \colon X \to Y_1$ is purely inseparable at $x$, and $\Psi_{\rm et} \colon Y_1 \to Y$ is \'etale
at $y_1 = \Psi(x)$. Assume further that $Y$ is smooth at $y$. Then, once again, $\Frak(A) \cap \widehat{\mathcal{O}}_x = A$.
\end{lemma}

\begin{proof}  (a) If $\Psi$ is \'etale at $x$, then
$\Psi^* \colon \widehat{\mathcal{O}}_y \to \widehat{\mathcal{O}}_x$ is an isomorphism.
Hence, $A = \widehat{\mathcal{O}}_x$, and $\Frak(A)\cap \widehat{\mathcal{O}}_x = \Frak(A) \cap A = A$.

(b) We may assume without loss of generality that $\operatorname{char}(K) = p > 0$; otherwise
$\Psi_{\rm insep}$ is the identity map, and part (b) reduces to part (a).
Decompose $\Psi^* \colon \widehat{\mathcal{O}}_y \twoheadrightarrow A \hookrightarrow \widehat{\mathcal{O}}_x$ 
as
\DeclareRobustCommand\longtwoheadrightarrow
     {\relbar\joinrel\twoheadrightarrow}
\[ \Psi^* \colon \widehat{\mathcal{O}}_y \; \stackrel{\; \Psi_{\rm et}^* \;}{\longrightarrow} \;
\widehat{\mathcal{O}}_{y_1} \stackrel{\; \Psi_{\rm insep}^* \;}{\longtwoheadrightarrow} A \hookrightarrow \widehat{\mathcal{O}}_x. \] 
Now $\Psi_{\rm et}^*$ is an isomorphism between
$\widehat{\mathcal{O}}_y$ and $\widehat{\mathcal{O}}_{y_1}$, as in part (a). 
Consequently, $A = \Psi_{\rm insep}^*(\widehat{\mathcal{O}}_{y_1})$.
After replacing $Y$ by $Y_1$ and $y$ by $y_1$, we see that for the purpose of proving part (b), we may replace $\Psi$ by
$\Psi_{\rm insep}$. In other words, we may assume without loss of
generality that $\Psi = \Psi_{\rm insep} \colon X \to Y$ is a purely inseparable
morphism, i.e., $\widehat{\mathcal{O}}_x$ 
is purely inseparable over $A$. Then for any $u/v \in \Frak(A) \cap \widehat{\mathcal{O}}_x$, 
there exists an $r \geqslant 0$ such that $(u/v)^{p^r}\in A$. Note that as a $k$-algebra, 
\[ A \simeq \widehat{\mathcal{O}}_y \simeq K[[t_1, \ldots, t_{\dim(Y)} ]]  \]
where the first isomorphism is via $\Psi^*$ and the second follows from our assumption that $Y$ is smooth at $y$.
Hence, $A$ is a unique factorization domain. In particular, $A$ integrally closed and 
thus $u/v \in A$. This completes the proof of Lemma~\ref{lem3.1}.
\end{proof}

\begin{lemma}
\label{lem3.2}
Let $G$ be a semiabelian variety and
$\Psi \colon G \to G$ be a dominant self-map of varieties over $K$ (not necessarily a group homomorphism). Then

\smallskip
(a) $\Psi = T_{\alpha} \circ \Psi_0$, where $T_{\alpha} \colon G \to G$ denotes translation by some $\alpha \in G(K)$ and
$\Psi_0 \colon G \to G$ is a dominant group homomorphism.

\smallskip
(b) Moreover, there exists another semiabelian variety $G_1$, a purely inseparable group homomorphism
$\Psi_{\rm insep} \colon G \to G_1$ and an etale morphism of varieties $\Psi_{\rm et} \colon G_1 \to G$
such that $\Psi = \Psi_{\rm et} \circ \Psi_{\rm insep}$.
\end{lemma}

\begin{proof} (a) is proved in~\cite[Theorem 2]{Iitaka}. To prove (b), write $\Psi = T_{\alpha} \circ \Psi_0$ as in part (a).
Denote the kernel of $\Psi_0$ by $N$; it is a finite abelian group scheme. Its connected component $N^0$ is 
infinitesimal, and the component group $N/N^0$ is \'etale; see~\cite[Exp.~VI$_A$, Proposition 5.5.1]{sga}.
Thus $\Psi_0$ can be written as a composition $\Psi_1 \circ \Psi_2$, where $\Psi_2 \colon G \to G/N^0$ is purely inseparable,
and $\Psi_1 \colon G/N^0 \to G/N \simeq G$ is \'etale. Setting $G_1 = G/N^0$, $\Psi_{\rm et}= T \circ \Psi_1$
and $\Psi_{\rm insep} = \Psi_2$, we obtain a desired decomposition $\Psi = \Psi_{\rm et} \circ \Psi_{\rm insep}$. 
\end{proof}

\begin{lemma}
\label{lem:quotient}
Let $G$ be a semiabelian variety defined over an algebraically closed field $K$, endowed with a regular dominant self-map $\Psi$. Let $W\subset G$ be a totally invariant irreducible subvariety under the action of $\Psi$ and let $W_0:={\rm Stab}_G(W)$ be the stabilizer of $W$ under the translation action of $G$ on itself. We let $\iota:G\lra G/W_0$ be the natural quotient map. Then 

\smallskip
(a) $\Psi$ induces a dominant regular self-map $\overline{\Psi}$ on $\overline{G} := G/W_0$ such that 
$\overline{\Psi}\circ \iota = \iota\circ \Psi$.

\smallskip
(b) $\overline{W}:=\iota(W)$ is totally invariant under the action of $\overline{\Psi}$.

\smallskip
(c) The stabilizer of $\overline{W}$ is $\overline{G}$ is trivial.
\end{lemma}

\begin{proof}
Write $\Psi = T_{\alpha} \circ \Psi_0$ as in Lemma~\ref{lem3.2}(a).

(a) We claim that
\begin{equation} \label{claim:W_0_inv}
\Psi_0(W_0)\subseteq W_0. 
\end{equation}
To prove the claim, note that $\alpha + \Psi_0(W)=\Psi(W) = W$ and thus 
\begin{equation}
\label{eq:W_0_inv_2}
W_0={\rm Stab}_G(W)={\rm Stab}_G\left(\alpha + \Psi_0(W)\right)={\rm Stab}_G\left(\Psi_0(W)\right).
\end{equation}
On the other hand, since $\Psi_0$ is a group endomorphism and $W_0$ is the stabilizer of $W$, we have
\begin{equation}
\label{eq:W_0_inv_3}
\Psi_0(W_0)\subseteq {\rm Stab}_G\left(\Psi_0(W)\right).
\end{equation}
Claim~\eqref{claim:W_0_inv} readily follows from~\eqref{eq:W_0_inv_2} and \eqref{eq:W_0_inv_3}. 

Continuing with the proof of part (a), 
in view of~\eqref{claim:W_0_inv}, $\Psi$ induces a regular self-map $\overline{\Psi} \colon \overline{G} \to \overline{G}$ given by
\[ \overline{\Psi}(\iota(x)):=\iota (\alpha) +\iota (\Psi_0(x)) . \]
Since $\Psi$ is dominant, $\overline{\Psi}$ is also dominant. 

\smallskip
(b) Choose $x\in G$ with the property that $\overline{\Psi}(\iota(x))\in \overline{W}$.
Then there exists $w_0 \in W_0$ such that $w_0 +\Psi(x)\in W$. Translating both sides by $-w_0$ and remembering that $W_0$ is the stabilizer of $W$, we see that $\Psi(x)\in W$.  Since $W$ is totally invariant under the action of $\Psi$, we conclude that $x \in W$
and thus $\iota(x) \in \overline{W}$. This shows that $\overline{W}$ is totally invariant under the action of $\overline{\Psi}$. 

\smallskip
(c) Suppose $\iota(g)$ lies in the stabilizer of $\overline{W}$. Our goal is to show that $g \in W_0$.
Indeed, for any $w \in W$, we have $\iota(g) + \iota(w) \in \overline{W}$ or equivalently, $g + w \in W + W_0$.
By the definition of $W$, $W + W_0 = W$. Thus $g \in \operatorname{Stab}_G(W) = W_0$, as claimed.
\end{proof}

\section{Proof of Proposition~\ref{cor:indeterminacy}}
\label{sec:indeterminacy}

\subsection{Key lemma}

Our proof of Proposition~\ref{cor:indeterminacy} will rely on the following lemma.

\begin{lemma}
\label{lem:J_nonempty}
Let $G$ be a semiabelian variety defined over an algebraically closed field $K$. Suppose $f \circ \Psi= g$ for some dominant regular self-map $\Psi \colon G \to G$ and some dominant rational self-maps
$f, g \colon G \dasharrow G$. Then $g$ is regular at a point $x \in G(K)$ if and only if $f$ is regular at $y = \Psi(x)$.
\end{lemma}

\begin{proof}
One direction is obvious: if $f$ is regular at $y$, then $g$ is regular as $x$. To prove the opposite implication, assume that
$g$ is regular at $x$ and set $z:=g(x)$. Consider the diagram of induced maps
\[ \xymatrix{\widehat{\mathcal{O}}_y  
\ar@{->}[dr]_{\Psi^*}    &  \widehat{\mathcal{O}}_z \ar@{->}[d]^{g^*} \\                                     
& \widehat{\mathcal{O}}_x } 	\]
where $\widehat{\mathcal{O}}_x$, $\widehat{\mathcal{O}}_y$ and $\widehat{\mathcal{O}}_z$ denote the completions of the local ring of $G$ at $x$, $y$ and $z$, respectively.
Our goal is to prove that $f^*:\widehat{\mathcal{O}}_z \lra\widehat{\mathcal{O}}_y$ is well defined. This will tell us
that $f$ is regular at $y$. A priori, we only know that $f$ induces an inclusion
$ f^* \colon \Frak(\widehat{\mathcal{O}}_z) 
\hookrightarrow 
\Frak(\widehat{\mathcal{O}}_y)$
of the fraction field of $\widehat{\mathcal{O}}_z$ into the fraction field of $\widehat{\mathcal{O}}_y$.

Assume the contrary: there exists an $\alpha\in \widehat{\mathcal{O}}_z$ such that 
\begin{equation}
\label{eq:contrary}
f^*(\alpha)=\frac{\beta}{\gamma}\notin \widehat{\mathcal{O}}_y \text{ for some }\beta,\gamma\in \widehat{\mathcal{O}}_y.
\end{equation}
Applying $\Psi^*$ to both sides of~\eqref{eq:contrary} and remembering that $\Psi^* \circ f^* = g^*$, we obtain 
\begin{equation}
\label{eq:contrary_2}
\frac{\Psi^*(\beta)}{\Psi^*(\gamma)}= g^*(\alpha)\in \widehat{\mathcal{O}}_x.
\end{equation}
By Lemma~\ref{lem3.2}, $\Phi$ can be decomposed as $\Psi_{\rm et} \circ \Psi_{\rm insep}$, where
$\Psi_{\rm insep} \colon G \to G_1$ is an inseparable algebraic group homomorphism for some semiabelian variety $G_1$ and
$\Psi_{\rm et} \colon G_1 \to G$ is an etale map. By Lemma~\ref{lem3.1}(b), 
\begin{equation} \label{e3.2}
\Frak(A) \cap \widehat{\mathcal{O}}_x = A.
\end{equation}
\smallskip
By~\eqref{eq:contrary_2}, $\Psi^*(\beta/\gamma) = {\Psi^*(\beta)}/{\Psi^*(\gamma)}$ 
lies in both $\widehat{\mathcal{O}}_x$ and $\Frak(A)$.
By~\eqref{e3.2}, $\Psi^*(\beta/\gamma)=\Psi^*(\alpha_0)$ for some $\alpha_0\in \widehat{\mathcal{O}}_y$.
Now recall that since $\Psi \colon G \to G$ is dominant, $\Psi^*$ is injective.
Thus 
\[ \frac{\beta}{\gamma}=\alpha_0\in \widehat{\mathcal{O}}_y, \]
contradicting~\eqref{eq:contrary}. This contradiction completes the proof of Lemma~\ref{lem:J_nonempty}.
\end{proof}

\begin{remark}  \label{ex:matrix}
Lemma~\ref{lem:J_nonempty} fails if 
we replace $G$ by an arbitrary smooth variety (not necessarily a semiabelian variety).   
For example, let $\mathbb K$ be a field of characteristic $0$ and consider  
the rational map $f \colon \mathbb{A}^2\dashrightarrow \mathbb{A}^2$ given by 
$(x,y) \mapsto (x^{10} y^2, x^2 y^{-1})$.  The map $f$ is not regular, but 
$f^2$ and $f^3$ and both regular.
To prove this, set
\[ C:= \left( \begin{array}{cc} 10 & 2 \\ 2 & -1\end{array}\right), \] 
compute $C^2$ and $C^3$ and check that these matrices have strictly positive entries. This gives, in fact, that $f^n$ is regular for every $n\ge 2$, despite the fact that $f$ itself is not regular. 

Letting $\Psi = f^2$, $g = f^3$, $x = (0, 0)$, and $y = \Psi(x) = (0, 0)$, we see that $f \circ \Psi = g$, 
$\Psi$ is regular, $g$ is regular at $x$ but $f$ is not regular at $y$.
Thus Lemma~\ref{lem:J_nonempty} fails here.
Note that the regular map $\Psi$ in this example is not \'etale at $x$, and thus Lemma~\ref{lem3.1} does not apply.
\end{remark}

\subsection{Conclusion of the proof of Proposition~\ref{cor:indeterminacy}} 
By~\cite[Proposition 1.3]{artin} the indeterminacy locus of a rational map from a nonsingular variety into a group variety is pure of codimension $1$ (possibly empty). Thus in the setting of Proposition~\ref{cor:indeterminacy}, the indeterminacy locus $W_n$ of $\Phi^n$ is pure of codimension $1$ (again, possibly empty) in $G$. Since $\Phi^{n+m}=\Phi^m \circ \Phi^n$ and $\Phi^m$ is regular, we clearly have $W_{n+m}\subseteq W_n$. 
By the Noetherian property, the descending chain
\[ W_1 \supseteq W_{m+1} \supseteq W_{2m + 1} \supseteq W_{3m + 1} \supseteq \ldots \]
of closed subset of $G$ terminates. That is, 
there exists an integer $k \geqslant 0$ such that
$W_{mk+1} = W_{m i + 1}$ for any $i\geqslant k$. In particular,
$\Phi^{km +1}$ and $\Phi^{(k+1)m + 1}$ have the same indeterminacy locus, $W = W_{km + 1} = W_{(k+1)m + 1}$. On the other hand, by 
Lemma~\ref{lem:J_nonempty} with $\Psi = \Phi^m$,
$f = \Phi^{km + 1}$ and $g = \Phi^{(k+1)m + 1}$, we see that
$\Phi^{km + 1}$ is regular at $x$ if and only if $\Phi^{(k+1)m + 1}$
if regular at $\Phi^m(x)$. In other words, 
$W_{(k+1)m + 1} = \Phi^{-m}(W_{km + 1})$ or equivalently
$W = \Phi^{-m}(W)$.

Now recall that $W$ is a finite union of distinct irreducible subvarieties $W_1, \ldots, W_r$ of $G$ of codimension $1$. 
If $r = 1$, i.e., $W$ is irreducible, then we have constructed a $W$ with desired properties, and the proof is complete.

In general, note that by Lemma~\ref{lem3.2}(a), $\Phi^m \colon G \to G$ is surjective on $K$-points.
Since $W$ is totally invariant, $\Phi^m$ restricts to a surjective map $W \to W$. In particular, 
$\Phi^m$ permutes the generic points of the irreducible components $W_1, \ldots, W_r$. After 
replacing $m$ by a multiple $md$, where $d$ is the order of this permutation,
we may assume that $\Psi = \Phi^m \colon G \to G$ 
maps each $W_i$ dominantly onto itself.
We claim that $W_1$ is totally invariant under $\Psi$. If we establish this claim, then after
replacing $W$ by $W_1$, our the proof of Proposition~\ref{cor:indeterminacy} will be complete.

To prove the claim, note that $\Psi^{-1}(W_1) = W_1 \cup V_2 \cup V_3 \cup \ldots \cup V_r$, where
$V_i = \Psi^{-1}(W_1) \cap W_i$. Since $\Psi$ maps each $W_i$ dominantly onto itself, $V_i \subsetneq W_i$
for each $i = 2, \ldots, r$. In particular, each $V_i$ is of codimension $\geq 2$ in $G$. On the other hand, by Lemma~\ref{lem3.2}(b),
$\Psi \colon G \to G$ is the composition of an \'etale morphism and a purely inseparable morphism. This tells us that $\Psi^{-1}(W_1)$
is pure of codimension $1$ in $G$. We conclude that $\Psi^{-1}(W_1) = W_1$ (in other words, $V_2, \ldots, V_r$ are all contained in $W_1$),
as desired.
\qed

\begin{remark} \label{rem.general-base-field} The same proof shows that Proposition~\ref{cor:indeterminacy} remains valid 
over an arbitrary field $K$, if we interpret ``irreducible" as ``irreducible over $K$". In general the codimension $1$ subvariety
$W$ in the statement of Proposition~\ref{cor:indeterminacy} may not be irreducible over the algebraic closure $\overline{K}$ of $K$.
\end{remark}


\section{Proof of Proposition~\ref{prop:invariant}}
\label{sec:proof}

We note that throughout the proof we may, without loss of generality, replace $\Psi$ by a conjugate $U\circ \Psi \circ U^{-1}$,
where $U$ is an automorphism $G \to G$ (e.g., a translation).

\subsection{The case where $G$ is of dimension one}

\begin{claim}
\label{lem:G_1}
Proposition~\ref{prop:invariant} holds when ${\rm dim}(G)=1$.
\end{claim}

\begin{proof}[Proof of Claim~\ref{lem:G_1}.]
A $1$-dimensional semiabelian variety $G$ is isomorphic to either $\bG_m$ or to an elliptic curve, and a totally invariant 
irreducible codimension $1$ subvariety $W$ of $G$ is a single point $x_0\in G$. After replacing $\Psi$ 
by $T_{x_0}^{-1}\circ \Psi\circ T_{x_0}$, 
where $T_{x_0} \colon G \to G$ is translation by $x_0$, we may assume that $\Psi$ is a group endomorphism 
of $G$ and $x_0 = 0$ is the identity element of $G$. The condition that $W = \{ 0 \}$ is totally invariant is equivalent to saying that $\operatorname{Ker}(\Psi)$
is an infinitesimal subgroup of $G$.

If $G=\bG_m$ then every endomorphism $\Psi$ is given by $\Psi(t) = t^d$. The condition that the kernel of $\Psi$ is infinitesimal translates to $d = \pm 1$ if $\operatorname{char}(K) = 0$ and $d = \pm p^r$ for some $r \geqslant 0$ if $\operatorname{char}(K) = p > 0$.
Setting $\chi \colon \bG_m \to \bG_m$ to be the identity map, we see that $\chi \circ \Psi^2 = \chi$  in characteristic $0$
and $\chi \circ \Psi^2 = F^{2r} \circ \chi$ in characteristic $p$, as desired.

If $G$ is an elliptic curve and $\Psi$ is separable, then $\Psi$ must be an automorphism of $G$. Since the automorphism group of
$G$ is finite, we have $\Psi^k = \operatorname{id}$ for some $k \geqslant 1$. Thus taking $\chi \colon G \to G$ to be the identity map, 
we obtain $\chi \circ \Psi^k = \chi$, as desired. In particular, this proves Claim~\ref{lem:G_1} in charcateristic $0$.

Finally, if $G$ is an elliptic curve, $\operatorname{char}(K) = p > 0$, then we can write
$\Psi$ as the composition of a separable isogeny with the $n$th power of the Frobenius map $F^n$ for some $n \geqslant 0$. 
Let $G^{(p^n)}$ be the image of the elliptic curve $G$ under $F^n$. Then  $$\Psi=\tau\circ F^n,$$
for some positive integer $n$, where $F^n \colon G\lra G^{(p^n)}$ and $\tau \colon G^{(p^n)}\lra G$ is a separable isogeny. Since $\operatorname{Ker}(\Psi)$ is infinitesimal, we conclude that $\tau$ must be an automorphism of $G$. 
Hence, $G$ is isomorphic with $G^{(p^n)}$. In particular, $G$ is isomorphic to an elliptic curve $G_1$ defined over a finite field, 
and an iterate of $\Psi$ induces a power of the corresponding Frobenius endomorphism on $G_1$. Thus the conclusion of Proposition~\ref{prop:invariant}(b) holds in this case as well.
\end{proof}

So, from now on, we may assume $\dim(G)>1$, which in particular means that the totally invariant subvariety $W$ has positive dimension.

\subsection{Reducing to the case of a subvariety with trivial stabilizer}

Next, we observe that for the purpose of proving Proposition~\ref{prop:invariant}, we may assume that
\begin{equation} \label{eq:stabilizer-trivial}
\operatorname{Stab}_{G}(W) = \{ 0 \}.
\end{equation}
Indeed, by Lemma~\ref{lem:quotient}, $\Psi \colon G \to G$ descends to
$\overline{\Psi} \colon \overline{G} \to \overline{G}$, via the natural projection
$\iota \colon G \to \overline{G} = G/W_0$. Moreover, $\overline{W} = \iota(W)$ is a codimension $1$
irreducible subvariety of $\overline{G}$ which is totally invariant under $\overline{\Psi}$
and $\operatorname{Stab}_{\overline{G}}(\overline{W_0}) = 0$. Suppose we know that
Proposition~\ref{prop:invariant}(a) holds for $\overline{W}$, $\overline{\Psi}$ and $\overline{G}$. In other words, there
exists a non-constant homomorphism $\overline{\chi} \colon \overline{G} \to G_1$ such that $\overline{\chi} \circ \overline{\Psi}^k =
\overline{\chi}$ for some $k \geqslant 1$
Then composing $\chi$ with $\iota$, we obtain a homomorphism $\chi \colon G \to G_1$ such that $\chi \circ \Psi^k = \chi$.
Hence, Proposition~\ref{prop:invariant} also holds for $W$, $G$ and $\Psi$.
Similarly, in part (b), if we know that there exists a semiabelian variety $G_1$ defined over a finite field $\mathbb F_q$ and a homomorphism
$\overline{\chi} \colon \overline{G} \to G_1$ such that $\chi \circ \Psi^k = F^r \circ \overline{\chi}$, then once again, 
composing $\chi$ with $\iota$, we obtain a homomorphism $\chi \colon G \to G_1$ such that $\chi \circ \Psi^k = F^r \circ \chi$.

\subsection{Working under the assumption the invariant subvariety has positive dimension and trivial stabilizer}

From now on we will assume that~\eqref{eq:stabilizer-trivial} holds. We will see that this is a very strong assumption; in particular, when $K$ is of characteristic zero it forces an iterate of $\Psi \colon G \to G$ to be the identity map.
The rest of the proof of Proposition~\ref{prop:invariant} will rely on the theorem of Pink and Rossler~\cite[Theorem~3.1]{P-R}.
Assumption~\eqref{eq:stabilizer-trivial} simplifies the conclusion of this theorem and will thus
be crucial for our argument. 

Write $\Psi=T_\alpha\circ \Psi_0$, where $\Psi_0$ be a dominant group endomorphism of $G$ and $\alpha\in G(K)$, 
as in Lemma~\ref{lem3.2}(a). By~\cite[Theorem~3.1]{P-R} we can write 
\begin{equation}
\label{eq:def_W}
W=\gamma+ h_1(X_1) + h_2(X_2) + \ldots + h_l(X_l)
\end{equation}
where

\begin{itemize}
\item $\gamma \in G(K)$,

\smallskip
\item $A_1, \ldots, A_l$ are semiabelian varieties, each equipped with 
dominant homomorphisms $f_i \colon A_i\lra A_i$,

\smallskip
\item
$X_i$ is a $f_i$-invariant closed subvariety of $A_i$ for each $i = 1, \ldots, l$, and

\smallskip
\item
$h_{i} \colon A_{i} \to G$ is a homomorphism of semiabelian varieties satisfying
\[ \Psi_0\circ h_i=h_i\circ f_i \text{ for each $i=1,\dots, l$.} \]
\end{itemize}

Furthermore, each dynamical system $\left(A_{i},f_i\right)$ can be assumed to be pure of some weight $a_i\geqslant 0$ in the sense of \cite[Definition~2.1]{P-R}.  

At the cost of replacing $\Psi$ and each $f_i$ an iterate, we may assume each $a_i$ is a non-negative integer. Moreover, 
if $a_i = 0$, we may assume that $f_i$ is the identity homomorphism on $A_i$ (again, after replacing $\Psi$ by a suitable iterate). 
and consequently, $\Psi_0$ induces the identity homomorphism on $G_i = h_i(A_i) \subset G$. 

On the other hand, $a_i$ can only be a positive integer if $\operatorname{char}(K) = p > 0$, $A_i$ is a semiabelian variety 
defined over a finite subfield of $K$ and $f_i$ is the $a_i$-th power of the usual Frobenius endomorphism. 
Clearly we may assume that each $A_{i}$ is nontrivial. Finally, by ~\cite[Theorem~3.1]{P-R} we know that the map $h_1\times \cdots \times h_l : A_1\times \cdots \times A_l \lra G$ is an isogeny since its image is a semiabelian variety containing a translate of $W$ (see equation~\eqref{eq:def_W}) and moreover, $W$ is a codimension-$1$ subvariety of $G$ with trivial stabilizer (see assumption~\eqref{eq:stabilizer-trivial}). 
It is precisely at this point where we use the assumption that $\dim(G) > 1$ or equivalently, that $\dim(W) > 0$. 

Next we split our analysis into two cases. In the first case $a_1 = \cdots = a_l = 0$, in the second case one of the integers $a_i$ is positive.
In particular, in characteristic $0$, only the first case can occur.

\begin{claim}
\label{lem:identity}
If $a_1 = \ldots = a_l = 0$, then $\Psi$ is the identity map.
In particular, the conclusion in Proposition~\ref{prop:invariant}(a) holds.
\end{claim}

\begin{proof}[Proof of Claim~\ref{lem:identity}.]
Since the induced action of $\Psi_0$ on each $G_i$ is the identity and furthermore, $G=\sum_{i=1}^l G_i$, we conclude that $\Psi_0$ must be the identity map on $G$. Then $\Psi = T_{\alpha}$ is the translation map by some point $\alpha\in G$. Since $W$ is a $\Psi$-invariant subvariety with trivial stabilizer, this means that $\alpha=0$. Therefore, $\Psi$ is the identity map. In this case the conclusion of Proposition~\ref{prop:invariant} is obvious: we can that $G_1 = G$ and $\chi \colon G \to G_1$ to be the identity map.
\end{proof}

\begin{claim}
\label{lem:non-identity}
If there is some $i\in\{1,\dots, l\}$ such that $a_i > 0$, then $A_1$ is defined over a finite subfield $\F_q$ of $K$, and there exists 
a homomorphism $\chi \colon G \to A_1$ of semiabelian varieties such that $\chi \circ \Psi = F^{a_1} \circ \chi$, where $F$ is the Frobenius map.
\end{claim}

\begin{proof}[Proof of Claim~\ref{lem:non-identity}.]
We may assume without loss of generality that $i = 1$. As mentioned above, since 
$a_1 > 0$ we may further assume that $\operatorname{char}(K) = p > 0$, $A_1$ is defined over a finite field $\F_q$, and
$f_1 \colon A_1 \to A_1$ is the $a_1$-th power of the Frobenius map.

Write $\Psi =  T_{\alpha} \circ \Psi_0$, as in Lemma~\ref{lem3.2}(a).
Moreover, $G= G_1 + \cdots + G_l$, where $G_i= h_i(A_i)$. Choose $\alpha_i\in G_i$ so that
\[ \alpha=\alpha_1 + \cdots + \alpha_l. \]
The choice of the $\alpha_i$'s is not unique. On the other hand, the sum map
\begin{equation}
\label{eq:301}
s \colon G_1 \times \cdots \times G_l \lra G \text{ given by }s(x_1,\dots, x_l)= x_1 + \ldots + x_l
\end{equation}
is finite.  By the definition of the $G_i$'s, $\Psi_0$ induces a group homomorphism $(\Psi_0)_{|G_i}$
on each $G_i$. By a slight abuse of notation we will simply refer to this group homomorphism
as $\Psi_0$. Furthermore, for each $i=1,\dots,l$, the induced action of $\Psi$ on $G_i$ 
is given by 
\begin{equation}
\label{eq:302}
\Psi(x):=\Psi_0(x)+\alpha_i \text{ for each }x\in G_i.
\end{equation}

Since $\Psi_0\circ h_1=h_1\circ f_1$ and $f_1$ is the $a_1$-th power of the Frobenius (with $a_1>0$), we see that
$\left(f_1-{\rm id}_{A_1}\right):A_1\lra A_1$ is a dominant map. Thus $(\Psi_0)|_{G_1}:G_1\lra G_1$ has the property that 
$$
\left((\Psi_0)|_{G_1}-{\rm id}_{G_1}\right):G_1\lra G_1\text{ is also a dominant map.}
$$
Therefore, we can find $\beta_1\in G_1(K)$ such that 
$\left((\Psi_0)|_{G_1}-{\rm id}_{G_1}\right)(\beta_1)=\alpha_1$ or equivalently,
$\Psi_0(\beta_1) = \beta_1 + \alpha_1$. Setting
$\Psi' = T_{\beta_1}\circ \Psi\circ T_{\beta_1}^{-1}$, and remembering~\eqref{eq:302},
we see that for every $x \in G_1$,
\[ \Psi'(x) = \beta_1 + \Psi(x - \beta_1) = \beta_1 +\Psi_0(x - \beta_1) + \alpha_1 =
\beta_1 + \Psi_0(x) - \Psi_0(\beta_1) + \alpha_1 = \Psi_0(x). \]
In other words, after replacing $\Psi$ by a conjugate, $\Psi'$, 
may assume that 
\begin{equation}
\label{eq:105}
\text{$\Psi(x) = \Psi_0(x)$ for every $x \in G_1$.}
\end{equation} 

Fix some $i\in \{1,\dots, l\}$. The discussion before \cite[Definition~2.1,~p.~774]{P-R} shows that
there exists another $\Psi_0$-equivariant isogeny $j_i \colon G_i\lra A_i$. 
Here ``$\Psi_0$-equivariant" means that $f_i\circ j_i=j_i\circ \Psi_0$. Moreover, 
the isogeny $j_i$ has the property that $j_i\circ h_i=[m_i]_{A_i}$ is the multiplication 
map by some positive integer $m_i$ on $A_i$. Combining these, we obtain an  isogeny 
$$ G_1 \times \cdots \times G_l \to A = A_1 \times \cdots \times A_l $$ 
given by $j:= \left(m_0j_1, \ldots, m_0j_l \right)$. 
Since $\operatorname{Ker}(j)$ factors through $\operatorname{Ker}(s)$, there exists an isogeny
$\tilde{s} \colon G \lra A = A_1 \times \ldots \times A_l$ such 
that $j=\tilde{s}\circ s$. Concretely, for a given $x\in G(K)$, if we write 
\begin{equation}
\label{eq:101}
x=\sum_{i=1}^l x_i\text{ for some }x_i\in G_i;
\end{equation}
then,
\begin{equation}
\label{eq:102}
\tilde{s}(x):= \left(m_0j_1(x_1),\dots, m_0j_l(x_l)\right).
\end{equation}
The choice of $x_1, \ldots, x_l$ in~\eqref{eq:101} is not unique but $\tilde{s}(x)$ is independent of this choice. 
Furthermore, since $j$ is $\Psi_0$-equivariant, the map $\tilde{s} \colon G\lra A$ is also $\Psi_0$-equivariant. 
In other words, $f \circ \tilde{s}=\tilde{s}\circ \Psi_0$, where $f:=(f_1,\dots,f_l) \colon A \to A$.
Projecting to the first factor, $\pi_1 \colon A\lra A_1$, and setting $\chi=\pi_1\circ \tilde{s}:G\lra A_1$, we see that
\begin{equation}
\label{eq:104}
\chi \circ \Psi_0 = f_1 \circ \chi.
\end{equation}
It remains to show that $\chi \circ \Psi_0 \colon G \to A_1$ on the left hand side of~\eqref{eq:104} can be replaced 
by $\chi \circ \Psi$. In other words, 
\begin{equation} \label{e.last-step}
\text{$\chi \circ \Psi_0(x) = \chi \circ \Psi(x)$ for every $x \in G$.}
\end{equation}
Once this equality is established,~\eqref{eq:104} can be rewritten as $\chi \circ \Psi_0 = f_1 \circ \chi$.
Since $f_1$ is the $a_1$th power of the Frobenius, this will complete the proof of Claim~\ref{lem:non-identity} and thus of
Proposition~\ref{prop:invariant}.

To prove~\eqref{e.last-step}, write $x = x_1 + \ldots + x_l$, where each $x_i$ lies in $X_i$, as in~\eqref{eq:101}. Then 
\begin{align*} 
\chi \circ \Psi(x) & =   \pi_1 \circ \tilde{s} (\Psi(x_1) + \Psi(x_2) + \ldots + \Psi(x_s)) \\ 
 & =  \pi_1 \circ \tilde{s}(\Psi(x_1) + \Psi(x_2) + \ldots + \Psi(x_l)) \\
 &  =  \pi_1 \big( m_0 j_1(\Psi(x_1)), m_0 j_2(\Psi(x_2)), \ldots, m_0 j_l(\Psi(x_l)) \big) \\
 & =  m_0 j_1(\Psi(x_1))
\end{align*}
and similarly
$ \chi \circ \Psi_0(x) = m_0 j_1(\Psi_0(x_1))$.
By~\eqref{eq:105}, $\Psi(x_1) = \Psi_0(x_1)$, and~\eqref{e.last-step} follows.
This completes the proof of~\eqref{e.last-step} and thus of Claim~\ref{lem:non-identity},
Proposition~\ref{prop:invariant} and Theorem~\ref{thm:main}.
\end{proof}




\end{document}